\documentclass[12pt,a4paper]{amsart}

  \usepackage{url,tikz,tikz-cd}
  \usetikzlibrary{calc}
  \usetikzlibrary{matrix}

  \usepackage{bbm, wasysym, amssymb}  
  \usepackage{verbatim, hyperref}
  \usepackage{graphicx}
  \usepackage{wrapfig}
  \usepackage{epsfig}
  \usepackage{epsf}

  \usepackage{color}

  \newcommand{\supp}{\ensuremath{\mathrm{supp}}}

  \newcommand{\N}{\ensuremath{\mathbb{N}}}

  \newcommand{\C}{\ensuremath{\mathbb{C}}}

  \newcommand{\udots}{\mathinner{\mskip1mu\raise1pt\vbox{\kern7pt\hbox{.}}
      \mskip2mu\raise4pt\hbox{.}\mskip2mu\raise7pt\hbox{.}\mskip1mu}}

  \renewcommand{\epsilon}{\varepsilon}
  \renewcommand{\phi}{\varphi}

  \newtheorem{conjecture}{Conjecture}[section]
  
  \newtheorem{example}[conjecture]{Example}
  \newtheorem{lemma}[conjecture]{Lemma}
  \newtheorem{observation}[conjecture]{Observation}
  \newtheorem{proposition}[conjecture]{Proposition}
  \newtheorem{corollary}[conjecture]{Corollary}
  
  \newtheorem{remark}[conjecture]{Remark}
  \newtheorem{definition}[conjecture]{Definition}
  
   \newtheorem{theorem}[conjecture]{Theorem}

\begin{document}
  \newcommand{\OptSubscript}[1][]{_{#1}}
  \newcommand{\OptExponent}[1][]{^{#1}}
  \newcommand{\AddParentheses}[1]{(#1)}
  \newcommand{\AddBrakets}[1]{[#1]}
  \newcommand{\ExpPar}[1][]{\OptExponent[#1]\AddParentheses}
  \newcommand{\ExpBra}[1][]{\OptExponent[#1]\AddBrakets}
  \newcommand{\SubExp}[1][]{\OptSubscript[#1]\OptExponent}
  \newcommand{\SubExpPar}[1][]{\OptSubscript[#1]\ExpPar}
  \newcommand{\SubExpBra}[1][]{\OptSubscript[#1]\ExpBra}
  \makeatletter
  \newcommand{\newvariable}[2]{%
    \newcommand{#1}{#2\SubExp}
    \expandafter\newcommand\csname\expandafter\@gobble\string#1Of\endcsname{#2\SubExpPar}
  }
  \makeatother
  \newvariable{\AffBuild}{X}
  \newvariable{\Func}{\mathcal{F}}
  \newvariable{\TheOperator}{\operatorname{D}}
  \newvariable{\TheWidth}{w}
  \newcommand{\notion}[1]{\emph{#1}}
  \newvariable{\mapcolon}{:}
  \newvariable{\TheField}{\mathbb{K}}
  \newvariable{\TheVertex}{v}
  \newvariable{\TheFunction}{u}
  \newvariable{\Ball}{\mathbb{B}}
  \newvariable{\FinBW}{\mathcal{B}}
  \newvariable{\TheGroup}{G}
  \newvariable{\TheGroupElement}{g}
  \newvariable{\TheBaseVertex}{{\TheVertex[0]}}
  \newvariable{\TheVertexSet}{V}
  \newvariable{\TheCore}{\Lambda}
  \newvariable{\Dim}{\operatorname{dim}}
  \newvariable{\Card}{\operatorname{card}}
  \newcommand{\TheStabilizer}{\TheGroup[0]}
  \newvariable{\DimFunc}{p}
  \newvariable{\TildeA}{\tilde{\mathsf{A}}}
  \newvariable{\StdApartment}{\Sigma}
  \newvariable{\StdChamber}{c}
  \newvariable{\TheVertexType}{t}
  \newvariable{\StdVertex}{w}
  \newcommand{\IfEmptyThenElse}[1]{%
    \begingroup
    \def\dummy{#1}%
    \def\empty{}%
    \ifx\dummy\empty
      \def\next##1##2{##1}%
    \else
      \def\next##1##2{##2}%
    \fi
    \expandafter\endgroup\next
  }
  \newcommand{\SetOf}[2][]{\IfEmptyThenElse{#1}{
      \left\{\,#2\,\right\}
    }{
      \left\{\,#1\, : \, #2 \,\right\}
    }
  }
  \newcommand{\FamOf}[2][]{\IfEmptyThenElse{#1}{
      \left(\,#2\,\right)
    }{
      \left(\,#2\,\right)_{#1}
    }
  }
  \newvariable{\Sl}{\operatorname{SL}}
  \newvariable{\TheRank}{n}
  \newvariable{\TypeA}{\begin{color}{blue}0\end{color}}
  \newvariable{\TypeB}{\begin{color}{red}1\end{color}}
  \newvariable{\TypeC}{\begin{color}{black}2\end{color}}
  \newvariable{\CoxGroup}{W}
  \newvariable{\CoxGen}{S}
  \newvariable{\CoxElement}{w}
  \newvariable{\AltCoxElement}{v}
  \newvariable{\TheChamber}{c}
  \newvariable{\AltChamber}{d}
  \newvariable{\WDist}{\delta}
  \newvariable{\TheRadius}{n}
  \newvariable{\TheBall}{B}
  \newvariable{\NatNumbers}{\mathbb{N}}
  \newvariable{\FinSizeOps}{\mathcal{B}}
  \newvariable{\CharFct}{\mathbbm{1}}
  \newvariable{\MeanCharFct}{\tilde{\CharFct}}
  \newvariable{\TheIndex}{i}
  \newvariable{\TheIndexSet}{I}
  \newvariable{\TheCoefficient}{a}
  \newvariable{\MeanFunction}{\tilde{\TheFunction}}
  \newvariable{\TheRepresentation}{\rho}
  \newvariable{\SymGroup}{\operatorname{Sym}}
  \newvariable{\ThePerm}{\pi}
  \newvariable{\ImRep}{Q}
  \newvariable{\GAvgFunction}{\overline{\TheFunction}}
  \newvariable{\PreImRep}{\overline{\ImRep}}
  \newvariable{\AvgOperator}{\overline{\TheOperator}}
  \newvariable{\Image}{\operatorname{im}}
  \newcommand{\avg}{\operatorname{\mathbb{E}}\nolimits}
  \newvariable{\TheReflection}{s}
  \newvariable{\Thickness}{q}
  \newcommand{\crossprod}{\times}
  \newcommand{\WordLengthOf}[1]{|#1|}
  \newcommand{\pref}[1]{(\ref{#1})}
  \newvariable{\Identity}{\operatorname{Id}}
  \newvariable{\TheFieldElement}{x}
  \newvariable{\IntNumbers}{\mathbb{Z}}
  \newvariable{\Oka}{\mathcal{O}}
  \newvariable{\MaxIdeal}{\mathfrak{m}}
  \newvariable{\UnitValuation}{\pi}
  \newcommand{\UnitsOf}[1]{{#1}^*}
  \newvariable{\TheValuation}{v}
  \newcommand{\epirightarrow}{\rightarrow
    \negthinspace\negthinspace\negthinspace\negthinspace
    \rightarrow}
  \newvariable{\TheVectorSpace}{V}
  \newvariable{\TheLattice}{L}
  \newvariable{\AltLattice}{{L'}}
  \newvariable{\TheScalar}{\alpha}
  \newcommand{\HomothetyOf}[1]{[#1]}
  \newvariable{\TheClass}{\Lambda}
  \newvariable{\AltClass}{{\Lambda'}}
  \newvariable{\Gl}{\operatorname{GL}}
  \newvariable{\StdVect}{\mathbf{e}}
  \newvariable{\TheTree}{X}
  \newvariable{\Center}{Z}
  \newvariable{\TheGeodesic}{g}
  \newvariable{\TheInteger}{m}
  \newvariable{\TheVector}{\mathbf{e}}
  \newvariable{\AltVector}{\mathbf{f}}
  \newvariable{\TheMatrix}{A}
  \newcommand{\NormOf}[1]{\| #1 \|}
  \newcommand{\ExpOf}[1]{\mathrm{e}^{#1}}
  \newcommand{\SpanOf}[1]{\langle #1 \rangle}
  \newcommand{\ProjLineOf}[1]{\mathbb{P}_1(#1)}
  \newcommand{\BoundaryOf}[1]{\partial #1}
  \newvariable{\ParGeodesic}{\gamma}
  \newvariable{\RealNumbers}{\mathbf{R}}
  \newvariable{\TheTime}{t}
  \newvariable{\TheGraph}{\Gamma}
  \newvariable{\Pg}{\mathbf{PG}}
  \newvariable{\TheShift}{s}

  \title[Poisson Transforms for Trees]{Poisson Transforms for Trees of Bounded Degree}
  \author{K.-U. Bux, J. Hilgert, T. Weich}
  
\address{K.-U. Bux: Fakult\"at f\"ur Mathematik, Universit\"at Bielefeld, Postfach 100131, D-33501 Bielefeld, Germany; \email{}{bux@math.univ-bielefeld.de}}
  
\address{J. Hilgert: Institut f\"ur Mathematik, Universit\"at Paderborn, D-33095 Paderborn, Germany; \email{}{hilgert@math.univ-paderborn.de}}

\address{T. Weich: Institut f\"ur Mathematik, Universit\"at Paderborn, D-33095 Paderborn, Germany; \email{}{weich@math.univ-paderborn.de}
}

\begin{abstract} 
We introduce a parameterized family of Poisson transforms on trees of bounded degree, construct explicit inverses for generic parameters, and characterize moderate growth of Laplace eigenfunctions by H\"older regularity of their  boundary values.
\end{abstract}
  \maketitle
  \thispagestyle{empty}


\tableofcontents

\section{Introduction}
The classical Poisson transform is an integral transform from the circle to the unit disk turning functions on the circle into harmonic functions on the disk. The transform is injective and the function on the circle can be recovered as the boundary value of the harmonic function.

Harmonic functions are defined as functions annihilated by the classical Laplace operator. The  hyperbolic Laplace-Beltrami operator on the disk, viewed as the Poincar\'e disk, agrees with the classical Laplace operator up to multiplication with a nowhere zero scalar function, so harmonicity  is equivalent to harmonicity with respect to the Laplace-Beltrami operator. This observation is the starting point of a rich theory of Poisson transforms for Riemannian symmetric spaces of non-compact type, which yields joint eigenfunctions for the algebra of invariant differential operators on such spaces. 

In this paper we study Poisson transforms for trees, which in the homogenous situation correspond to symmetric spaces of constant negative curvature, and give a simple proof for the well-known fact (see \cite[Thm.~1.2]{FTN91} and \cite[Thm.~A]{MZ83}) that generic Poisson transforms are a bijection between finite additive measures on the boundary and eigenfunctions of the tree Laplacian. The geometric approach we take allows us to completely remove the homogeneity condition on the tree while at the same time still give explicit inverses. But note that in the case of non-homogeneous trees the Poisson transformation is a bijection between finitely additive measures and the kernel of the Laplacian plus some explicitly given potential. In the case of homogeneous trees the potential is simply a constant and one recovers the classical result of \cite{MZ83,FTN91}, which we thus have extended to an analog of spaces of negative but variable curvature. 

A closely related result on Poisson transformations for general trees has recently been obtained by Anantharaman and Sabri \cite{AS19}. Their perspective, however, is complementary to ours: They start with a given graph and a Schr\"odinger operator and construct a suitable Poisson kernel in terms of Green's functions. In our paper we fix the Poisson kernel to have the usual completely explicit form in terms of horocycle brackets.

Another new consequence that we draw from our explicit construction of the boundary value map concerns the regularity of boundary values: 
In the symmetric space context, one knows that the Poisson transformations relate regularity properties of generalized functions on the boundary to growth conditions of eigenfunctions on the space. The most general domain for the Poisson transform are hyperfunctions on the boundary, resulting in eigenfunctions without further restrictions \cite{K+78}. Moderate growth of eigenfunctions  then correponds to distributions on the boundary, i.e. elements of the dual space of the space of smooth functions \cite{OS80, BS87}. In the case of trees the role of hyperfunctions is taken by finitely additive measures. We show that for trees of bounded degree moderate growth of eigenfunctions corresponds to the dual of certain Banach spaces of H\"older continuous functions. Our motivation to work out this result is not just to have a tree analogon of \cite{OS80, BS87}; the H\"older dual spaces on the boundary that correspond to eigenfunctions of weak moderate growth are precisely the functional spaces that appear in the spectral theory of transfer operators for subshifts of finite type (see e.g. \cite{Ba00}). In a subsequent work we want to apply the regularity results of this paper to establish a quantum-classcial correspondence on graphs, analogous to what has been achieved on Riemannian symmetric spaces \cite{DFG15, KW19, GHW21, HHW21}.
 
The paper is structured as follows. In Section~\ref{sec:trees} we introduce trees and their boundaries, which are compact totally disconnected spaces. In Section~\ref{sec:FA-measures} we introduce finitely additive measures on compact totally disconnected spaces and study them in detail for boundaries of trees. In Section~\ref{sec:Laplace} we recall the definition of the Laplace operator on trees and introduce the parameterized family $(P_z)_{z\in\C}$ of Poisson transforms mapping finitely additive measures on the boundary to functions on the vertices of the tree satisfying a generalized eigenvalue equation for the Laplace operator. For $z^2\not\in\{0,1\}$ we construct a ``boundary value map'' $\vec\beta_z$ in the inverse direction which inverts $P_z$ (Thm.~\ref{thm:Poisson}). It turns out that in the case of homogeneous trees for large enough $|z|$ we can recover $\mu$ as a weak limit from the values $P_z(\mu)(x)$ for $x$ tending to infinity (Cor.~\ref{thm:bdy-value}). In Section~\ref{sec:Hoelder} we introduce the relevant H\"older spaces and prove  our characterization of Laplace eigenfunctions of moderate growth (Thm.~\ref{thm:mod-growth}).  

\emph{Acknowledgements} We thank Etienne Le Masson for helpful remarks. TW acknowledges support
from the Deutsche Forschungsgemeinschaft (DFG) (Grant No. WE 6173/1-1  Emmy Noether group ``Microlocal Methods for Hyperbolic Dynamics'').

\section{Trees of bounded degree and their boundaries}\label{sec:trees}

Let $(\mathfrak X,\mathfrak E)$ be a graph. We view $\mathfrak E$ as a set of two-element subsets of $\mathfrak X$. For $e=\{a,b\}$ with $a,b\in\mathfrak X$ precisely $a$ and $b$ are \emph{incident} with $e$. We also consider the directed version $(\vec{\mathfrak X},\vec{\mathfrak E})$ with $\vec{\mathfrak X}=\mathfrak X$ and $\vec{\mathfrak E} =\{(a,b)\in\mathfrak X^2\mid e=\{a,b\}\in \mathfrak E \}$. For $\vec e=(a,b)$ we call $a=:\iota(\vec e)$ the \emph{initial} and $b=:\tau(\vec e)$ the \emph{terminal point} of $\vec e$. Thus we obtain two maps $\iota, \tau:\vec{\mathfrak E}\to \mathfrak X$. Graphs carry a natural metric $d$ given by the minimal \emph{length} (number of edges) of \emph{chains} (i.e. paths that do not backtrack) between two vertices.

We call $\mathfrak T:= (\mathfrak X, \mathfrak E)$ a \emph{tree} if the graph is connected and if there are no circular chains. On trees for any two vertices $x$ and $y$ there is a unique chain $[x,y]$ connecting $x$ and $y$ and we have  $d(x,y)=\ell([x,y])$. From $\mathfrak X$ and the metric $d$ one can reconstruct the entire tree. Given a vertex $x$ we define the degree $\text{deg}(x)$ at $x$ to be the number of vertices $y$ such that $\{x,y\}\in \mathfrak E$. For convenience we write $q_x:=\text{deg}(x)-1$ and we call the tree to be of \emph{bounded degree} if there is $q_x\leq q_\text{max}<\infty$. From now on we will always assume that all our trees are of bounded degree. 

The \emph{boundary at infinity} $\Omega$ of a tree $(\mathfrak X,\mathfrak E)$ is the set of equivalence classes $[(x_j)_{j\in\mathbb N_0}]$ of infinite chains $(x_j)_{j\in\mathbb N_0}$ of vertices, where two such chains are called \emph{equivalent} if they share infinitely many vertices.

\begin{remark}{\rm
Our definition of a tree does not exclude vertices of degree $1$ as they show up for instance in each finite tree. Geometrically, one might want to view such vertices as part of the boundary. The reason we do not do that is that while our setup is meaningful even for finite trees, our main results become void if there are no boundary points at infinity. So from now on we will only call the elements of $\Omega$ boundary points of $\mathfrak T$.
}
\end{remark}

The disjoint union $\mathfrak X\coprod \Omega$ carries a natural compact topology such that $\Omega$ is a compact subset. This topology is characterized by the fact that each point of $\mathfrak X$ is open and for $\omega\in \Omega$ a basis of neighborhoods of $\omega$ in $\mathfrak X\cup \Omega$ is formed by the sets 
$$\mathfrak X(x,y):=\big\{\omega'\in\Omega\ \big|\  [y,\omega'[\cap [x,y] =\{y\}\big)\}\cup\big\{z\in [y,\omega'[{}\ \big| \ [y,\omega'[\cap [x,y] =\{y\}\big\}$$
with $x\in\mathfrak X$ and $y\in [x,\omega[$, the chain starting in $x$ and defining $\omega$. Then the relative topology on $\Omega$ consists of the sets
$$\Omega_x(y):=\{ \omega\in\Omega\mid y\in [x,\omega[\}$$
for $x,y\in\mathfrak X$. Given $x\in \mathfrak X$ and $n\in\mathbb N$,  $\Omega= \bigcup_{d(x,y)=n} \Omega_x(y)$ is a disjoint union of open compact sets (see \cite[p.~5]{FTN91}). This implies that $\Omega$ as a topological space is totally disconnected.

\section{Finitely additive measures}\label{sec:FA-measures}

In this section we introduce finitely additive measures on compact totally disconnected spaces and study them in some detail in the case of the boundary of a tree.

We start with some general observations on locally constant functions. To this end we fix a locally compact Hausdorff space $Z$.

\begin{remark}\label{Omega-remark}
{\rm  We denote the space of locally constant functions $p:Z\to V$ with values in some $\C$-vectorspace $V$ by $C^{\mathrm{lc}}(Z,V)$ and set $C_c^{\mathrm{lc}}(Z,V):= C_c(Z)\cap C^{\mathrm{lc}}(Z,V)$.
For all $p\in C_c^{\mathrm{lc}}(Z,V)$ we have 
$$p=\sum_{v\in V}  \mathbf{1}_{p^{-1}(v)} v,$$
where $\mathbf{1}_U$ is the indicator function of $U$.
}
\end{remark}

If $Z$ is discrete, the condition ``locally constant'' is void.

\begin{definition}\label{def:alg-duals} We let $\mathcal K'(Z)$ be the algebraic dual of $C_c^{\mathrm{lc}}(Z)$ and $\mathcal K_c'(Z)$ the algebraic dual of $C^{\mathrm{lc}}(Z)$. 
\end{definition}

Note that $\mathcal K'(Z)=\mathcal K_c'(Z)$ if $Z$ is compact.

\begin{proposition}\label{prop:finite image}
For a  continuous map $p: Z\to \C$ with compact support the following conditions are equivalent:
\begin{itemize}
\item[(1)] $p$ is locally constant.
\item[(2)] $p$ takes only finitely many values.
\end{itemize}
The implication $(2)\Rightarrow(1)$ holds also without the compact support assumption.
\end{proposition}

\begin{proof}
Suppose that $p$ is locally constant. Then the compactness of $\supp(p)$ implies that the locally constant function $p|_{\supp(p)}$ takes only finitely many values. Conversely, if a continuous function $p: Z\to \C$ takes only finitely many values, then $p^{-1}(z)$, which is closed as $\{z\}$ is closed, is always open. Thus, $p$ is locally constant. 
\end{proof}

Note that a continuous function $p: Z\to \C$ with only finitely many values need not have compact support unless $Z$ itself is compact. Just consider the constant function $1$.

\begin{remark}\label{rem:K'S+X}{\rm
Let $\lambda\in\mathcal K'(\mathfrak X\times \Omega)$ and $u\in C^\mathrm{lc}\big(\mathfrak X, \mathcal K'(\Omega)\big)=C\big(\mathfrak X, \mathcal K'(\Omega)\big)$. Then the formulas
\begin{eqnarray*}
\forall F\in C_c^\mathrm{lc}(\mathfrak X\times \Omega):\quad \langle \lambda_u,F\rangle
&:=& \sum_{x\in \mathrm{supp} (F^\Omega)}\langle u(x), F^\Omega(x)\rangle\\
\forall x\in \mathfrak X, f\in C^\mathrm{lc}(\Omega):\quad \langle u_\lambda(x),f\rangle
&:=&\langle \lambda, \delta_x\otimes f\rangle,
\end{eqnarray*}
where $F^\Omega(x)(\omega):=F(x,\omega)$ and $\delta_x(y)=1$ if $y=x$ and $0$ otherwise, define mutually inverse maps between  $\mathcal K'(\mathfrak X\times \Omega)$ and $C^\mathrm{lc}\big(\mathfrak X, \mathcal K'(\Omega)\big)$. In fact:

\begin{eqnarray*}
\langle \lambda_{u_\lambda},F\rangle
&=&\sum_{y\in\mathrm{supp}(F^\Omega)}\langle u_\lambda(y),F^\Omega(y)\rangle
\ =\ \sum_{y\in\mathrm{supp}(F^\Omega)}\langle \lambda,\delta_y\otimes F^\Omega(y)\rangle \\
&=&\Bigg\langle \lambda,\sum_{y\in\mathrm{supp}(F^\Omega)}\delta_y \otimes F^\Omega(y)\Bigg\rangle 
\ =\ \left\langle \lambda,F\right\rangle 
\end{eqnarray*}
and, in view of $(\delta_x\otimes f)^\Omega(y)= \delta_x(y)f$,
$$\langle u_{\lambda_u}(x),f\rangle\ =\ \langle \lambda_u,\delta_x\otimes f\rangle\ =\sum_{y\in \mathrm{supp}(\delta_x)} \langle u(y), \delta_x(y)f\rangle \ =\ \langle u(x), f\rangle.$$

}
\end{remark}

Next we turn to finitely additive measures on compact totally disconnected spaces. We fix such a space and denote it by $Z$.

\begin{definition}[Finitely additive measures]{\rm
Let $\Sigma$ denote the set of \emph{clopen}, i.e. open and closed, subsets of $Z$. 
A \emph{finitely additive measure} is a map $\mu:\Sigma\to \mathbb C$ such that
\begin{itemize}
\item[(a)] $\mu(\emptyset)=0$.
\item[(b)] $\forall U,U'\in \Sigma:\ \mu(U\cup U')+\mu(U\cap U')= \mu(U)+\mu(U')$.
\end{itemize}
We denote the space  of finitely additive measures on $Z$ by $\mathcal{M}_{\mathrm{fa}}(Z)$. 
}
\end{definition}

Note that (b) is equivalent to 
\begin{itemize}
\item[(b')] $\forall U,U'\in \Sigma \text{ disjoint}:\ \mu(U\cup U')= \mu(U)+\mu(U')$.
\end{itemize}
This follows immediately by writing $U\cup U'=U\cup\big(U'\setminus(U\cap U')\big)$.

\begin{remark}[Clopen sets]\label{rem:clopensets}{\rm
For $U\subset Z$ we have
$$U\in\Sigma\quad \Longleftrightarrow\quad \mathbf{1}_U \in C^{\mathrm{lc}}(Z).$$

In fact, $U=\mathbf{1}_U^{-1}(1)$ and $U^c:=Z\setminus U=\mathbf{1}_U^{-1}(0)$, so continuity of  $\mathbf{1}_U$ implies that $U$ and $U^c$ are closed. Conversely, since $1$ and $0$ are the only possible values for $\mathbf{1}_U$ we also see that $U\in \Sigma$ implies that $\mathbf{1}_U$ is continuous with discrete image and hence locally constant. 
}
\end{remark}

\begin{proposition}\label{prop-Dual1}
For $\mu\in \mathcal{M}_{\mathrm{fa}}(Z)$ the map
$$\langle \mu, \bullet\rangle: C_c^{\mathrm{lc}}(Z)\to\mathbb C,\quad  p\mapsto \int_\Omega p\, d\mu:=\sum_{z\in \mathbb C} z\mu\big(p^{-1} (z)\big)$$
is well defined and linear. 
\end{proposition}

\begin{proof}
By Proposition~\ref{prop:finite image} the sum in the definition of $\langle \mu, \bullet\rangle$ is finite and the $p^{-1} (z)$ are in $\Sigma$. Thus $\langle \mu, \bullet\rangle$ is well-defined.

For the linearity let $p,q\in C^{\mathrm{lc}}(Z)$ and $\alpha,\beta\in\mathbb C$. Then
\begin{align*}
\langle \mu, \alpha p+\beta q\rangle
&= \sum_{z\in\mathbb C} z \mu \big((\alpha p+\beta q)^{-1} (z)\big)\\
&= \sum_{a,b\in\mathbb C} (\alpha a+\beta b) \mu \big(p^{-1} (a)\cap q^{-1}(b)\big)\\
&= \alpha\sum_{a,b \in\mathbb C} a \mu \big(p^{-1} (a)\cap q^{-1}(b)\big)+
\beta\sum_{a,b\in\mathbb C}  b \mu \big(p^{-1} (a)\cap q^{-1}(b)\big)\\
&= \alpha\sum_{a \in\mathbb C} a \mu \big(p^{-1} (a)\big)+
\beta\sum_{b\in\mathbb C}  b \mu \big( q^{-1}(b)\big)\\
&=\alpha\langle \mu,  p\rangle+ \beta\langle \mu,  q\rangle.
\end{align*}
\end{proof}

\begin{proposition}\label{prop:alg-dual-Omega} For  $\lambda\in \mathcal K'(Z)$ the map 
$$\mu_\lambda: \Sigma\to \mathbb C,\quad U\mapsto \lambda(\mathbf{1}_U)$$
is an element of $\mathcal{M}_{\mathrm{fa}}(Z)$.
\end{proposition}

\begin{proof}
As $\mathbf{1}_\emptyset =0$ we see that $\mu_\lambda(\emptyset)=\lambda(\mathbf{1}_\emptyset)=0$. Now suppose that $U, U'\in \Sigma$ are disjoint. Then $\mathbf{1}_{U\cup U'} = \mathbf{1}_U+\mathbf{1}_{U'}$, so that
$$\mu_\lambda(U\cup U')
=\lambda(\mathbf{1}_{U\cup U'})
=\lambda(\mathbf{1}_U)+\lambda(\mathbf{1}_{U'})
=\mu_\lambda(U)+\mu_\lambda(U').$$ 
\end{proof}

\begin{proposition}[Finitely additive measures as linear functionals]\label{prop:distributions}
The map $\mathcal K'(Z)\to \mathcal{M}_{\mathrm{fa}}(Z), \ \lambda\mapsto \mu_\lambda$ is a linear isomorphism.
\end{proposition}

\begin{proof} The map is well defined by Proposition~\ref{prop:alg-dual-Omega}. Its linearity is obvious. Apply Proposition~\ref{prop-Dual1} and Proposition~\ref{prop:finite image} to $\mu_\lambda$ to see that 
$$\langle\mu_\lambda, p\rangle
= \sum_{z\in\mathbb C} z \mu_\lambda\big((p^{-1}(z)\big) 
= \sum_{z\in\mathbb C} z \lambda(\mathbf{1}_{p^{-1}(z)}) 
= \lambda\Big(\sum_{z\in\mathbb C} z \mathbf{1}_{p^{-1}(z)}\Big) 
\overset{\text{\ref{Omega-remark}}}{=}\lambda(p)
$$
for $p\in C_c^{\mathrm{lc}}(Z)$.
Conversely, for $U\in \Sigma$ we have
$$\langle \mu, \mathbf{1}_U\rangle
=\sum_{z\in\mathbb C} z \mu\big((\mathbf{1}_U^{-1}(z)\big)
= \mu(U).$$
Thus $\mu\mapsto \langle \mu,\bullet\rangle:\mathcal{M}_{\mathrm{fa}}(Z) \to \mathcal K_{\mathrm{fin}}'(Z)$ and  $\lambda\mapsto \mu_\lambda: \mathcal K_{\mathrm{fin}}'(Z)\to \mathcal{M}_{\mathrm{fa}}(Z)$ are mutually inverse.   
\end{proof}

From now on we will identify $\mathcal K'(Z)$ and $\mathcal M_{\mathrm{fa}}(Z)$ for compact totally disconnected spaces. In particular, we have $\mathcal K'(\Omega)=\mathcal M_{\mathrm{fa}}(\Omega)$ for $\Omega$ being the boundary of our tree.

In order to describe  the nature of $\mathcal{M}_{\mathrm{fa}}(\Omega)$ 
in more detail we introduce the map $\partial_+:\vec{\mathfrak E}\to \mathcal P(\Omega)$ which associates with $\vec e \in \vec{\mathfrak E}$ all boundary points $\omega\in \Omega$  which can be reached through $\vec e $. Here $\mathcal P(\Omega)$ is the power set of $\Omega$. Reversing the orientation yields the map $\partial_-:\vec{\mathfrak E}\to \mathcal P(\Omega)$.

\begin{remark}\label{vecE-observation} We have the following disjoint decompositions of $\Omega$.
\begin{itemize}
\item[(i)] $\forall \vec e\in \vec{\mathfrak E}:\quad \Omega=\partial_+\vec e \cup \partial_- \vec e$.
\item[(ii)] $\forall x\in \mathfrak X:\quad \Omega=\bigcup_{\iota(\vec e) =x} \partial_+\vec e$. 
\end{itemize} 
\end{remark}

\begin{remark}\label{Sigma-observation} Recall the set $\Sigma\subseteq \mathcal P(\Omega)$ consisting of open and closed subsets of $\Omega$.
\begin{itemize}
\item[(i)] $\mathcal B:= \{\partial_+ \vec e\mid \vec e\in \vec{\mathfrak E}\}\subseteq \Sigma$ is a basis for the topology on $\Omega$.

\item[(ii)] $\emptyset, \Omega\in \Sigma$.

\item[(iii)] $U,U'\in\Sigma$ implies $U\cap U'\in \Sigma$ and $U\cup U'\in \Sigma$.
\end{itemize} 
\end{remark}

\begin{definition}{\rm
Let ${\vec e\,}^\text{op}\in \vec{\mathfrak E}$ denote the edge $\vec e\in \vec{\mathfrak E}$ with the opposite orientation.
$$L(\vec{\mathfrak E}):=\Big\{ F: \vec{\mathfrak E}\to \mathbb C\ \Big|\  \exists z\in \mathbb C, \forall x\in \mathfrak X,\forall \vec e \in \vec{\mathfrak E}:  F(\vec e)+F({\vec e\,}^{\mathrm{op}})=z= \sum_{\iota(\vec e\,') =x} F(\vec e\,')\Big\}.$$
}
\end{definition}
\begin{remark}{\rm
  We can describe the space $L(\vec{\mathfrak E})$ using only half of the edges and simplifying the local conditions. For that we fix a base point $o \in \mathfrak X$ and consider only edges pointing away from $o$. We put:
\begin{align*}
  \vec{\mathfrak E}_o & :=
  \{
  \vec e\in \vec{\mathfrak E}
  \mid \vec e \text{\ points away from } o
                        \}
  \\
  L( \vec{\mathfrak E}_o ) & :=
                             \Big\{ F: \vec{\mathfrak E}_o \to \mathbb C\ \Big|\ \forall \vec e \in \vec{\mathfrak E}_o: F( \vec e ) = \sum_{\iota({\vec e\,}') = \tau( \vec e)} F( {\vec e\,}' ) \Big\}                      
\end{align*}
Then restriction to $\vec{\mathfrak E}_o$ defines a linear isomorphism from $L(\vec{\mathfrak E})$ to $L(\vec{\mathfrak E}_o)$. To construct the inverse, one needs to figure out the value of $z\in\C$ from the definition of $L(\vec{\mathfrak E})$, but that has to be $z=\sum_{\iota(\vec e)=o}F(\vec e)$. From $z$, one finds the values for $F({\vec e\,}^\text{op})$ with $\vec e\in\vec{\mathfrak E}_o$.
}
\end{remark}

\begin{remark}{\rm
Composing $\partial_+$ with $\mu\in \mathcal{M}_{\mathrm{fa}}(\Omega)$ gives a map 
$$\vec{\mu}:\vec{\mathfrak E}\to \mathbb C,\quad \vec e \mapsto \mu\big(\partial_+(\vec e )\big).$$ 
It is our goal is to show that $\mu\to\vec{\mu}$ is linear bijection between  $\mathcal{M}_{\mathrm{fa}}(\Omega)$ and $L(\vec{\mathfrak E})$.
To that end we first note that Remark~\ref{vecE-observation} implies the following \emph{compatibility conditions} for edges at a vertex $x$ of $\mathfrak T$:

\begin{eqnarray}\label{eq:comp1}
\sum_{x=\iota(\vec e) } \vec{\mu} (\vec e) &=& \mu(\Omega)\\ \label{eq:comp2}
\vec{\mu} (\vec e) + \vec{\mu} ({\vec e\,}^\text{op} ) &=& \mu(\Omega).
\end{eqnarray}
Thus for each $\mu\in \mathcal{M}_{\mathrm{fa}}(\Omega)$ we have $\vec{\mu}\in L(\vec{\mathfrak E})$. 

A finitely additive measure $\mu\in \mathcal{M}_{\mathrm{fa}}(\Omega)$ is completely determined by its values on a basis of the topology. Therefore the map 
$\mathcal{M}_{\mathrm{fa}}(\Omega)\to L(\vec{\mathfrak E}),\ \mu\mapsto \vec{\mu}$
is in fact injective.
}
\end{remark}

\begin{theorem}\label{thm:fa-measures-Omega} The map
$\mathcal{M}_{\mathrm{fa}}(\Omega)\to L(\vec{\mathfrak E}), \ \mu\mapsto \vec{\mu}$
is a linear isomorphism.
\end{theorem}

\begin{proof}
It only remains to show that for each $f\in L(\vec{\mathfrak E})$ we can find a $\mu\in \mathcal{M}_{\mathrm{fa}}(\Omega)$ such that $\vec{\mu}=f$. 

Fix a base point $o\in \mathfrak X$ and define 
$$\vec{\mathfrak E}_o(U):= \{\vec e\in \vec{\mathfrak E}_o\mid \partial_+\vec e\subseteq U\}$$
for $U\in \Sigma$. Note that $\vec{\mathfrak E}_o(U)$ is partially ordered by
$$\vec e\le {\vec e\,}'\quad :\Longleftrightarrow\quad  \partial_+ \vec e\subseteq \partial_+ {\vec e\,}'.$$
Moreover, $\vec{\mathfrak E}_o(U)$ has a finite set $\max(\vec{\mathfrak E}_o(U))$ of  maximal elements and 
$$U=\bigcup_{\vec e\in \max(\vec{\mathfrak E}_o(U))} \partial_+ \vec e.$$ 
This union is disjoint since all $\vec e$ point away from $o$. We set 
$$\mu(U):= \sum_{\vec e\in \max(\vec{\mathfrak E}_o(U))} f(\vec e)$$
and note that $\mu$ is finitely additive. In fact, $\vec{\mathfrak E}_o(\emptyset) = \emptyset$ so that $\mu(\emptyset)=0$, and $U\cap U'=\emptyset$ implies $\vec{\mathfrak E}_o(U)\cap \vec{\mathfrak E}_o({U'})=\emptyset$. \\
 
\noindent
{\bf Claim}:  $\mu$ is independent of the choice of base point.\\ 

This claim proves the theorem since setting $o:=\iota(\vec e)$ for a fixed $\vec e\in \vec{\mathfrak E}$ gives $\vec{\mu}(\vec e)=\mu(\partial_+\vec e)=f(\vec e)$.

To prove the claim we may assume that $o$ and $o'$ are adjacent vertices and $\vec g\in \vec{\mathfrak E}$ points from $o$ to $o'$. Let $\vec h_1, \vec h_2, \ldots$ be the edges at $o'$ different from $\vec g^{\mathrm{op}}$ oriented in such a way that $o'=\iota(\vec h_j)$. Then we have $\partial_+ \vec g =\cup_j \partial_+ \vec h_j$ (disjoint union). Therefore we have
\begin{eqnarray*}
\vec g\in \max(\vec{\mathfrak E}_o(U))
&\Longleftrightarrow&
\vec g\in \vec{\mathfrak E}_o(U)\\
&\Longleftrightarrow&
\forall j:\ \vec h_j\in \vec{\mathfrak E}_{o'}(U)\\
&\Longleftrightarrow&
\forall j:\ \vec h_j\in \max(\vec{\mathfrak E}_{o'}(U)).
\end{eqnarray*}  
We decompose any $U\in \Sigma$ as $U= U_+\cup U_-$, where $U_\pm:= U\cap \partial_\pm \vec g$ and show that 
$\mu_o(U_\pm)=\mu_{o'}(U_\pm)$,
where $\mu_{o}$ denote the $\mu$ constructed from $f$ for the basepoint $o$. Then 
$$\mu_o(U)=\mu_o(U_+)+\mu_o(U_-)=\mu_{o'}(U_+)+\mu_{o'}(U_-)=\mu_{o'}(U)$$
and the claim follows after considering the following cases.

\begin{description}
\item[Case 1] $\partial_+ \vec g= U_+$. \\
In this case we have $\mu_o(U_+)= f(\vec g)$ and $\mu_{o'}(U_+)=\sum_j f(\vec h_j)$, but
$$f(\vec g)= \Big(f({\vec g\,}^{\mathrm{op}})+\sum_j f(\vec h_j)\Big)-f({\vec g\,}^{\mathrm{op}})= \sum_j f(\vec h_j).$$ 
\item[Case 2] $\partial_+ \vec g \not=U_+$. \\
In this case $\vec{\mathfrak E}_o({U_+})= \vec{\mathfrak E}_{o'}({U_+})$ and hence again $\mu_o(U_+)=\mu_{o'}(U_+)$.
\end{description}
The equalities $\mu_o(U_-)=\mu_{o'}(U_-)$ are shown in the same way. This concludes the proof.
\end{proof}

\begin{example}[Dirac measures]\label{ex:Dirac measure}{\rm
For $\omega\in \Omega$ the point evaluation $\mathrm{ev}_\omega: C^{\mathrm{lc}}(\Omega)\to \mathbb C, \ p\mapsto p(\omega)$ is linear, hence by Proposition~\ref{prop:alg-dual-Omega} defines a finitely additive measure.  It is the \emph{Dirac measure}
$$\delta_\omega: \Sigma\to \mathbb C,\quad U\mapsto \begin{cases} 1&\omega\in U,\\ 0& \omega\not\in U.\end{cases}$$ 
The corresponding function $\vec\delta_\omega$ is given by
$$\vec\delta_\omega(\vec e)=\begin{cases}1&\omega\in \partial_+(\vec e)\\
0&\omega\in\partial_-(\vec e).
\end{cases}
$$
}
\end{example}

\begin{example}[Rotation invariant measures for regular trees]\label{ex:invariant measure}{\rm
Suppose that $\mathfrak T$ is regular of degree $q+1$. Fix a point $x\in\mathfrak X$. Then the function $\vec \mu_{x}:\vec{\mathfrak E}\to \C$, which is defined by
$$\vec\mu_{x}(\vec e):=
\begin{cases} \frac{q^{-d(x,\iota(\vec e))}}{q+1}&\vec e\text{ points away from } x\\
1-\frac{q^{-d(x,\tau(\vec e))}}{q+1}&\vec e\text{ points to } x,
\end{cases}
$$
satisfies 
$\vec\mu_{x}(\vec e)+\vec\mu_{x}({\vec e\,}^{\mathrm{op}})=1$, since $\tau({\vec e\,}^{\mathrm{op}})=\iota(\vec e)$ and $\iota({\vec e\,}^{\mathrm{op}})=\tau(\vec e)$. By definition we find
$$\sum_{x=\iota(\vec e)} \vec\mu_{x}(\vec e)=(q+1)\frac{q^0}{q+1}=1,$$
and, for $x\not=y\in\mathfrak X$,
$$\sum_{y=\iota(\vec e)} \vec\mu_{x}(\vec e)
=q\frac{q^{-d(x,y)}}{q+1}  +\left(1-\frac{q^{-d(x,y)+1}}{q+1} \right)= 1.$$
Thus $\vec\mu_{x}\in L(\vec{\mathfrak E})$. It is clear that  $\vec\mu_{x}$ is invariant under all rotations of $\mathfrak T$ around $x$. It follows that the  corresponding measure $\mu_{x}\in \mathcal{M}_{\mathrm{fa}}(\Omega)$ is invariant under the induced ``rotations'' on $\Omega$. 
}
\end{example}

\section{Laplace eigenfunctions}\label{sec:Laplace}
The \emph{Laplacian} $\Delta$ on $\mathfrak T$ operates on functions $f:\mathfrak X\to\C$ and is given by
$$(\Delta f)(x)=\frac{1}{q_x+1}\sum_{\iota(\vec e)=x} f\big(\tau (\vec e)\big)
.$$

\begin{definition}{\rm For each function $\chi\in \operatorname{Maps}( \mathfrak X; \C )$ we denote the kernel  $\ker(\Delta-\chi)$ of the linear map $\Delta-\chi: \operatorname{Maps}( \mathfrak X; \C )\to \operatorname{Maps}( \mathfrak X; \C ),\ f\mapsto \Delta f-\chi f$ by $\mathcal E_\chi(\mathfrak X)$. If $\chi$ is constant this is simply the space of Laplace eigenfunctions with eigenvalue $\chi$.}
\end{definition}

From a physics perspective $\Delta-\chi$ looks like the Hamiltonian of a free quantum particle in a potential landscape described by $\chi$. We therefore will call $\chi$  the \emph{potential} in the sequel.

\begin{remark}[Poisson kernels and Poisson transforms]\label{ex:Poisson kernel}{\rm 
~
\begin{enumerate}
  \item[(i)] Fix a base point $o\in\mathfrak X$. For $\omega\in\Omega$ and $x\in\mathfrak X$, there exists a unique $y\in\mathfrak X$ such that $[o,\omega[{}\cap [x,\omega[{}=[y,\omega[{}$, and we set $\langle x,\omega\rangle:= d(o,y)-d(x,y)$. Thus, we have a map
    \[
    \langle\cdot,\cdot\rangle:\mathfrak X\times \Omega\to \mathbb Z.
    \]
    Note that for $x\in[o,\omega[{}$ we have $\langle x, \omega\rangle = d(o,x)$; and for $o\in[x,\omega[{}$ we find $\langle x, \omega\rangle = -d(o,x)$.
    
\item[(ii)] For fixed $x\in\mathfrak X$ the map $\langle x,\cdot\rangle:\Omega\to\mathbb Z$ is locally constant in $\omega\in\Omega$. In fact, for any $\vec e$ pointing away from $o$ and $x$, we have
$$\forall \omega,\omega'\in\partial_+(\vec e):\quad \langle x,\omega'\rangle = \langle x,\omega\rangle.$$

\item[(iii)] For any parameter $0\not=z\in\C$ the function $f_{z,\omega}:\mathfrak X\to\C,\ x\mapsto z^{\langle x,\omega\rangle}$ satisfies
\begin{eqnarray*}
(\Delta f_{z,\omega})(x)
&=&\frac{1}{q_x+1}\sum_{\iota(\vec e)=x} f_{z,\omega}\big(\tau(\vec e)\big)\\
&=&\frac{1}{q_x+1}\left(q_x z^{\langle x,\omega\rangle -1} +  z^{\langle x,\omega\rangle +1} \right)\\
&=& \frac{q_x z^{-1} + z}{q_x+1}f_{z,\omega}(x).
\end{eqnarray*}
Thus, if we define for any $z\in\C$ the potential 
$\chi(z): \mathfrak X\ni x\mapsto \frac{z+q_x z^{-1}}{q_x+1}$ then $f_{z,\omega}$ is in $\mathcal E_{\chi(z)}$. In the special case of a regular tree of degree $(q+1)$,  $\chi(z)=\frac{z+q z^{-1}}{q+1}$ is the constant function on the tree and $f_{z,\omega}$ is a Laplace eigenfunction with eigenvalue $\chi(z)$. 

\item[(iv)] The function $f_{z,\omega}$ closely related
with the Dirac measure $\delta_\omega$ from Example~\ref{ex:Dirac measure}: 
If $\vec e=(x,y)$ points toward $\omega$, i.e. $y\in[x,\omega[{}$, then  $f_{z,\omega}(y)=z^{\vec\delta_\omega(\vec e)}f_{z,\omega}(x)$. Thus $\vec \delta_\omega$ describes the growth of $f_{z,\omega}$ when moving in the direction of $\omega$. In particular, given one value of $f_{z,\omega}(o)$ one recovers $f_{z,\omega}$ from $\vec\delta_\omega$. 

\item[(v)] The potential $\chi(z)$ is independent of $\omega$, so we may build new elements in $\mathcal E_{\chi(z)}$ from the $f_{z,\omega}$ by taking linear combinations (keeping $z$ fixed). More generally, for any $\mu\in \mathcal M_{\mathrm{fa}}(\Omega)$ we can set
$$f_{z,\mu}(x):=\int_\Omega f_{z,\omega}\, \mathrm d\mu(\omega) = \langle \mu, f_{z,\bullet}(x)\rangle$$
in view of Proposition~\ref{prop-Dual1} and (ii). Then $f_{z,\mu}\in\mathcal E_{\chi(z)}(\mathfrak X)$. This construction explains why we call the map
$$p_z:\mathfrak X\times \Omega\to\C,\quad (x,\omega)\mapsto f_{z,\omega}(x)$$
the \emph{Poisson kernel} for the parameter $z\in\C$ and the map
$$P_z: \mathcal M_{\mathrm{fa}}(\Omega)\to \mathcal E_{\chi(z)}(\mathfrak X), \quad \mu\mapsto f_{z,\mu}=\int_\Omega p_z(\bullet,\omega)\,\mathrm d\mu(\omega)$$
the \emph{Poisson transform} for the parameter $z\in\C$.
\item[(vi)] $P_z\delta_\omega=f_{z,\omega}$. In fact,
\begin{eqnarray*}
(P_z\delta_\omega)(x)
&=&f_{z,\delta_\omega}(x)
\ =\ \int_\Omega f_{z,\nu}(x)\ \mathrm{d}\delta_\omega(\nu)
\ =\ \langle\delta_\omega, f_{z,\bullet}(x)\rangle
\ =\ f_{z,\omega}(x).  
\end{eqnarray*} 
\end{enumerate}
}
\end{remark}

Our goal in this section is to construct to construct an invervse
\[
  \vec\beta_z :  \mathcal E_{\chi(z)}(\mathfrak X) \rightarrow  \mathcal M_{\mathrm{fa}}(\Omega)
\]
for the Poisson transform $P_z$ and thereby to show that both maps are linear isomorphisms.

\begin{observation}\label{global-boundary-value}
  For $z^2\not\in\{0,1\}$, solving the condition
  \begin{equation}\label{beta}
    zf(\tau(\vec e)) - f(\iota(\vec e)) = (z^2-1)z^{d(o,\iota(\vec e))} \vec \mu(\vec e)\qquad\text{for all }\vec e \in {\mathfrak E}_o
  \end{equation}
  for $\vec \mu$ in terms of $f$ defines a linear map
  \[
    \beta_z : \operatorname{Maps}( \mathfrak X; \C ) \longrightarrow \operatorname{Maps}( {\mathfrak E}_o; \C ).
  \]
\end{observation}

\begin{lemma}\label{composition}
  For $z^2\not\in\{0,1\}$ and a finitely additive measure $\mu$ on $\Omega$, we have
  \[
    \beta_z( P_z( \mu ) ) = \vec\mu.
  \]
\end{lemma}
\begin{proof}
  We consider an oriented edge $\vec e\in {\mathfrak E}_o$ from the vertex $x=\iota(\vec e)$ to $y=\tau(\vec e)$. Note that
  \(
    \langle y, \omega \rangle = \langle x, \omega \rangle + 1
  \)
  if $\vec e$ points toward $\omega$ and
  \(
    \langle y, \omega \rangle = \langle x, \omega \rangle - 1
  \)
  otherwise. Thus, we obtain:
  \begin{align*}
    zP_z(\mu)(\tau(\vec e)) - P_z(\mu)(\iota(\vec e))
    & = \int_\Omega \left(z^{1+\langle\tau(\vec e),\omega \rangle} - z^{\langle\iota(\vec e),\omega\rangle}\right) \mathrm{d}\mu(\omega)
    \\
    & = \int_{\partial_+\vec e} (z^2-1) z^{\langle \iota(\vec e), \omega \rangle} \mathrm{d} \mu(\omega)
    \\
    & = (z^2-1) z^{d(o,\iota(\vec e))} \int_{\partial_+\vec e} \mathrm{d} \mu(\omega)
    \\
    & = (z^2-1) z^{d(o,\iota(\vec e))} \vec\mu( \vec e )
  \end{align*}
  It follows that $f=P_z(\mu)$ and $\vec\mu$ satisfy condition~(\ref{beta}).
\end{proof}

\begin{lemma}\label{previous}
  Assume $z^2\not\in\{0,1\}$. Let $f:\mathfrak X \rightarrow \C$ and $\vec \mu : {\mathfrak E}_o \rightarrow \C$ satisfy condition~(\ref{beta}), i.e., $\vec\mu=\beta_z(f)$. Then the following are equivalent:
  \begin{enumerate}
    \item
      The function $\vec \mu$ lies in $L({\mathfrak E}_o)$, i.e., it satisfies the compatibility conditions
      \[
      \vec\mu( \vec e ) = \sum_{\iota(\vec g) = \tau(\vec e)}{\vec\mu}(\vec g)\qquad\text{for all }\vec e \in {\mathfrak E}_o.
      \]
    \item
      The function $f$ solves $(\Delta-\chi(z))f=0$  everywhere except possibly at the vertex $o$, i.e., we have
      \[
      \sum_{v=\iota(\vec e)}f(\tau(\vec e)) = (z+q_vz^{-1})f(v)\qquad\text{for all }v\neq o.
      \]
  \end{enumerate}
\end{lemma}
\begin{proof}
  Note that for a vertex $v$, we have $v\neq o$ if and only if $v=\tau(\vec e)$ for a unique edge $\vec e$ pointing away from $o$. Let $u_0$ denote the initial vertex of that edge and let ${\vec e}_1,\ldots,{\vec e}_{q_v}$ denote the edges with initial vertex $v$ pointing even further away from $o$. Finally let $u_1,\ldots,u_{q_v}$ denote their respective terminal points. Then, we find
  \begin{align*}
    & \sum_{i=0}^{q_v}f(u_i) = (z + q_v z^{-1})f(v)
    \\ \Longleftrightarrow\quad & \sum_{i=1}^{q_v} ( f(u_i) - z^{-1}f(v) ) = zf(v) - f(u_0)
    \\ \Longleftrightarrow\quad & \sum_{i=1}^{q_v} \left( f(\tau({\vec e}_i)) - z^{-1}f(\iota({\vec e}_i) )\right) = zf(\tau(\vec e)) - f(\iota(\vec e) )
    \\ \Longleftrightarrow\quad & \sum_{i=1}^{q_v} (z^2-1)z^{d(o,v)}z^{-1}{\vec\mu}({\vec e}_i) = (z^2-1)z^{d(o,u_0)} {\vec\mu}(\vec e)
    \\ \Longleftrightarrow\quad & \sum_{i=1}^{q_v} {\vec\mu}({\vec e}_i) = {\vec\mu}(\vec e)             
  \end{align*}
  The claimed equivalence follows by letting $v$ range over all vertices other than $o$.
\end{proof}

\begin{proposition}\label{characterizing-eigenfunctions}
  Assume $z^2\not\in\{0,1\}$. Let $f:\mathfrak X \rightarrow \C$ and $\vec \mu : {\mathfrak E}_o \rightarrow \C$ satisfy Condition~(\ref{beta}), i.e., $\vec\mu=\beta_z(f)$. Then $f\in{\mathcal E}_{\chi(z)}(\mathfrak X)$ if and only if $\vec\mu\in L({\mathfrak E}_o)$ and $f(o)=\sum_{o=\iota(\vec e)}{\vec\mu}(\vec e)$.
\end{proposition}
\begin{proof}
  In view of Lemma~\ref{previous}, we only have to see that the local condition $f(o)=\sum_{o=\iota(\vec e)}{\vec\mu}(\vec e)$ is equivalent to $f$ being an eigenfunction of the Laplacian ``at $o$'' for the eigenvalue $\chi(z)$. This, again, is purely computational:
  \begin{align*}
    & \sum_{o=\iota(\vec e)}f(\tau(\vec e)) = (z+q_oz^{-1})f(o)
    \\ \Longleftrightarrow \quad & \sum_{o=\iota(\vec e)} ( f(\tau(\vec e)) - z^{-1} f(o) ) = zf(o) - z^{-1}f(o)
    \\ \Longleftrightarrow \quad & \sum_{o=\iota(\vec e)} ( zf(\tau(\vec e)) - f(o) ) = (z^2 - 1)f(o)
    \\ \Longleftrightarrow \quad & (z^2-1)\sum_{o=\iota(\vec e)} {\vec\mu}(\vec e) = (z^2 -1 )f(o)                    
  \end{align*}
  Division by the non-vanishing number $z^2-1$ finishes the proof.
\end{proof}

\begin{theorem}\label{thm:Poisson}
Assume $z^2\not\in\{0,1\}$. The linear map $\beta_z$ from Observation~\ref{global-boundary-value} restricts to a linear isomorphism ${\vec\beta}_z :  \mathcal E_{\chi(z)}(\mathfrak X) \rightarrow L({\mathfrak E}_o)$ making the following diagram commutative:
  \[
    \begin{tikzcd}
      \mathcal E_{\chi(z)}(\mathfrak X)
      \arrow{r}{\vec\beta_z}
      & L({\mathfrak E}_o) = L(\vec{\mathfrak E})  \\
      & \mathcal M_\mathrm{fa}(\Omega)
      \arrow[swap]{u}{\text{Thm.~\ref{thm:fa-measures-Omega}}}
      \arrow{ul}{P_z} 
    \end{tikzcd}
  \]
\end{theorem}
\begin{proof}
  That the restriction ${\vec\beta}_z : \mathcal E_{\chi(z)}(\mathfrak X) \rightarrow L({\mathfrak E}_o)$ has the desired range follows from Proposition~\ref{characterizing-eigenfunctions}. Commutativity of the diagram has been established in Lemma~\ref{composition}. It follows automatically that $\vec\beta_z$ is surjective.

  To see that ${\vec\beta}_z$ has trivial kernel, observe that a function $f:{\mathfrak X}\rightarrow\C$ lies in the kernel of $\beta_z$ if and only if
  \[
    zf(\tau(\vec e)) = f(\iota(\vec e))\qquad\text{for all }{\vec e} \in {\mathfrak E}_o.
  \]
  For $f\in{\mathcal E}_{\chi(z)}(\mathfrak X)$, we additionally find $f(o)=0$ by Proposition~\ref{characterizing-eigenfunctions}. It follows that $f$ vanishes everywhere by propagation along edges in ${\mathfrak E}_o$.
\end{proof}

Motivated by special cases many authors call the inverse of a Poisson transform a \emph{boundary value map}. This can be justified in special cases. We follow this tradition, our justification being Corollary~\ref{thm:bdy-value} below.

To simplify notation, we put
\(
  \Omega(x) := \Omega_o(x) =
  \{ \omega \in \Omega
  \mid
  x \in [o,\omega[ \}
\) for each vertex $x$. Note that $\Omega(o)=\Omega)$. For a
finitely additive measure $\mu$, we use $\mu(x)$ as shorthand for
$\mu(\Omega(x))$.

\begin{lemma}\label{subtrees}
  Fix a finitely additive measure $\mu$, a complex number $z\not\in\{-1,0,1\}$,
  and the Poisson transform $f=P_z(\mu)$.
  Let $x_0,x_1,x_2,\ldots$ be a chain towards $\omega$ and assume
  $\omega\in\Omega(x_0)$, i.e., the chain points away from the
  basepoint $o$. Let $m=d(o,x_0)$. Then, $d(o,x_k)=m+k$ and we have
  \begin{equation}\label{along-chain}
    \frac{
      f(x_k)
    }{
      z^{m+k}
    }
    =
    \frac{
      f(x_0)
    }{
      z^{m+2k}
    }
    +
    \frac{z^2-1
    }{
      z^2
    }
    \sum_{j=1}^{k}z^{2(j-k)} \mu(x_j)
  \end{equation}
\end{lemma}
\begin{proof}
  First, we obtain
  \[
    \frac{f(x_{j})}{z^{m+j}}
    =
    \frac{1}{z^2}\left(
      (z^2-1)\mu(x_j)
      +
      \frac{f(x_{j-1})}{z^{m+j-1}}
    \right)
  \]
  for each $j=1,2,\ldots,k$ by rearranging~(\ref{beta}). Now, the computation becomes
  a matter of back-substitution:
  \begin{align*}
    \frac{
      f(x_k)
    }{
      z^{m+k}
    }
    & =
    \frac{1}{z^2}\left(
      (z^2-1)\mu(x_k)
      +
      \frac{f(x_{k-1})}{z^{m+k-1}}
    \right)
      \\
    & =
    \frac{1}{z^2}\left(
      (z^2-1)\mu(x_k)
      +
      \frac{1}{z^2}\left(
      (z^2-1)\mu(x_{k-1})
      +
      \frac{f(x_{k-2})}{z^{m+k-2}}
      \right)
    \right)
    \\
    & =
    \frac{1}{z^2}\left(
      (z^2-1)\mu(x_k)
      +
      \frac{1}{z^2}\left(
      (z^2-1)\mu(x_{k-1})
      +\cdots+
      \frac{1}{z^2}\left(
      (z^2-1)\mu(x_1)
      +
      \frac{f(x_0)}{z^m}
      \right)\cdots
      \right)
    \right)
    \\
    & =
      \frac{z^2-1}{z^2}\mu(x_k)
      +
      \frac{z^2-1}{z^4}\mu(x_{k-1})
      +\cdots+
      \frac{z^2-1}{z^{2k-2}}\mu(x_1) + \frac{f(x_0)}{z^{m+2k}}
  \end{align*}
\end{proof}

\begin{remark}{\rm
    Assume that $|z| > 1$ and let $\delta_\omega$ be the Dirac measure for an end $\omega\in\Omega$.
    Then $\delta_\omega(x)=1$ for those vertices satisfying $x\in[0,\omega[$
    and $\delta_\omega(x)=0$ otherwise.
    Consider a chain $o=x_0,x_1,x_2,\ldots$. By~(\ref{along-chain}), we get
    \[
      \frac{f_{z,\omega}(x_k)}{z^k}
      =
      \frac{f(o)}{z^{2k}}
      +
      \frac{z^2-1}{z^2}
      \sum_{j=1}^{k} z^{2(j-k)} \delta_\omega(x_j)
      .
    \]
    If the chain  $o=x_0,x_1,x_2,\ldots$ defines $\omega$, all
    $\delta_\omega(x_j)=1$, and we find that
    $\frac{f_{z,\omega}(x_k)}{z^k}$ tends to $1$, the partial sum representing
    crowing pieces of the geometric series $1+z^{-2}+z^{-4}+z^{-6}+\cdots$.

    If the chain $o=x_0,x_1,x_2,\ldots$ does not define the end $\omega$,
    all but finitely many $\delta_\omega(x_j)$ vanish, and the non-vanishing
    entries get lower weights as $k$ increases. Hence in this case,
    $\frac{f_{z,\omega}(x_k)}{z^k}$ tends to $0$.

    Thus, we can recover the Dirac measure by passing to limits:
    \[
      \lim_{k \rightarrow \infty}
      \frac{f_{z,\omega}(x_k)}{z^k}
      =
      \begin{cases}
        1 & \text{if\ } x_0,x_1,x_2,\ldots \text{\ defines\ }\omega \\
        0 & \text{otherwise}
      \end{cases}
    \]
    This extends by linearity to all measures of finite support.
}
\end{remark}

\begin{theorem}\label{thm:regular}
  Let $\mu$, $z$, and $f$ be as in Lemma~\ref{subtrees} and assume
  that the tree $\mathfrak T$ is regular of degree $q+1$ with   $q<z^2$.
  Then for each vertex $x$ of distance $m=d(o,x)\geq 1$ to $o$, we have
  \[
    \mu(x)
    =
    \frac{z^2-q}{z^2-1}
    \lim_{k\rightarrow \infty}
    \frac{1}{z^{m+k}} \sum_{y\in S_k(x)}
    f(y)
  \]
  where $S_k(x)=\{ y \in \mathfrak X \mid d(x,y)=k, x\in[o,y] \}$ is
  the set of vertices at distance $k$ from $x$ away from $o$.
\end{theorem}
\begin{proof}
  Since $\mathfrak T$ is regular of degree $q+1$, we find that
  $S_k(x)$ has exactly $q^k$ elements. More precisely, the vertices
  in $S_0(x)\cup S_1(x)\cup\cdots\cup S_k(x)$ form a $q$-ary tree
  rooted at $x$
  with $S_j(x)$ as the vertex set at depth~$j$.
  In particular, for each vertex $y\in S_k(X)$ there is a unique
  chain $x=x_0(y),x_1(y),\ldots,x_k(y)=y$ with $x_j(y)\in S_j(x)$.
  Using~(\ref{along-chain}), we can then write
  \[
    \frac{f(y)}{z^{m+k}}
    =
    \frac{f(x)}{z^{m+2k}} + \frac{z^2-1}{z^2} \sum_{j=1}^{k} z^{2(j-k)}\mu(x_j(y))
  \]
  Summation over $S_k(x)$ yields:
  \begin{align*}
    \frac{1}{z^{m+k}}
    \sum_{y \in S_k(x)} f(y)
    & =
      \frac{f(x)}{z^m}\frac{q^k}{z^{2k}}
      +
      \frac{z^2-1}{z^2}
      \sum_{j=1}^{k}
      \sum_{y_j\in S_j(x)}
      \frac{q^{k-j}}{z^{2(k-j)}} \mu(y_j)
      \\
     & =
      \frac{f(x)}{z^m}\frac{q^k}{z^{2k}}
      +
      \frac{z^2-1}{z^2}
       \sum_{j=1}^{k}
       \frac{q^{k-j}}{z^{2(k-j)}}
       \sum_{y_j\in S_j(x)}
       \mu(y_j)
       \\
     & = 
      \frac{f(x)}{z^m}\frac{q^k}{z^{2k}}
      +
      \frac{z^2-1}{z^2}\mu(x)
       \sum_{j=1}^{k}
       \frac{q^{k-j}}{z^{2(k-j)}}
  \end{align*}
  The reason for the powers of $q$ is that there are $q^{k-j}$ chains through
  from $x$ to level $k$ through each vertex at level $j$.
  Also, note that
  $\mu(x) = \sum_{y\in S_j(x)}\mu(y)$ for any $j$.

  Now, the behavior as $k$ tends to infinity is clear as
  \(
    \sum_{j=1}^{k}
    \frac{q^{k-j}}{z^{2(k-j)}}
  \)
  limits to $(1-\frac{q}{z^2})^{-1}= \frac{z^2}{z^2-q}$
  whereas $\frac{q^k}{z^{2k}}$ tends to $0$. We obtain
  \[
    \lim_{k\rightarrow \infty}
    \frac{1}{z^{m+k}}
    \sum_{y \in S_k(x)} f(y)
    =
    \frac{z^2-1}{z^2-q} \mu(x)
  \]
  and the claim follows.
\end{proof}

\begin{definition}{\rm
For any clopen set $U\in\Sigma$ and $n\in\N_0$, we put
\(
  \mathfrak X_n(U)=\{x\in[o,U[{}\mid d(o,x)=n\}
  ,
\)
where
\(
  [o,U[{}=\bigcup_{\omega\in U}[o,\omega[
  .
\)
}
\end{definition}
Note that any clopen set $U$ is a finite union of sets $\Omega(x)$. In fact,
$U$ decomposes as the disjoint union
\[
  U = \bigcup_{x \in \mathfrak X_n(U)} \Omega(x)
\]
for any sufficiently large $n$.

Whereas the relationship between a finitely additive measure $\mu$ and
its Poisson transform $f=P_z(\mu)$ is algebraic, we can now see how to
recover $\mu$ from $f$ analytically by means of a limiting procedure in
the case of a regular tree.
\begin{corollary}\label{thm:bdy-value}
  Under the assumptions of Theorem~\ref{thm:regular}, for any clopen set
  $U\in\Sigma$ we have
  \[
    \mu(U) = \frac{z^2-1}{z^2-q} \lim_{n \rightarrow \infty}
    \frac{1}{z^n}
    \sum_{x \in \mathfrak X_n(U)} f(x)
    .
  \]
\end{corollary}
\begin{proof}
  It suffices to show the claim for the sets $\Omega(x)$ with $x\neq o$.
  In that case, however,
  the statement is just a restatement of Theorem~\ref{thm:regular}.
\end{proof}

\section{H\"older continuous functions}\label{sec:Hoelder}
In this section we give a characterization for the regularity of boundary values. 

Let $S^+\mathfrak X$ be the space of chains of the form $x=[x_0,\omega[{}  = (x_0,x_1,\ldots)$. For $0<\vartheta<1$ we define the metric $d_\vartheta(x,y):=\sum_{x_i\not=y_i}\vartheta^i$ on $S^+\mathfrak X$.

If we fix a base point $o\in\mathfrak X$ and $0<\vartheta <1$ we can also define a metric $d_{o,\vartheta}$ on the boundary $\Omega$ via
$$d_{o,\vartheta}(\omega_1,\omega_2):=\vartheta^{d_{\mathrm{max}}},$$
where $d_{\mathrm{max}}:= \sup\{d(o,v)\mid v\in [o,\omega_1[{}\cap [o,\omega_2[{}\}$. 

\begin{lemma}\label{lem:metric Omega}
The equivalence class of the metric $d_{o,\vartheta}$ on $\Omega$ does not depend on the choice of the base point $o\in\mathfrak X$. 
\end{lemma}

\begin{proof}
Let $o'\in\mathfrak X$ be another base point. For $\omega_1,\omega_2\in \Omega$ let $v,v'$ the two vertices realizing the maximal distance to $o$ in $[o,\omega_1[{}\cap [o,\omega_2[$, respectively $o'$ in $[o',\omega_1[{}\cap [o',\omega_2[$. Then there are only two possibilties: Either $v,v\in[o,o']$ or else $v,v'\not\in[o,o']$. In both cases we can check that 
$d(o,v)\le d(o',v')+d(o,o')$ and $d(o',v')\le d(o,v)+d(o,o')$ and obtain 
$$\vartheta^{d(o,o')} d_{o',\vartheta}(\omega_1,\omega_2)\le d_{o,\vartheta}(\omega_1,\omega_2)\le \vartheta^{-d(o,o')} d_{o',\vartheta}(\omega_1,\omega_2).$$
\end{proof}

\begin{definition}{\rm
On the space 
$$\mathcal F_\vartheta:=\{ f:S^+\mathfrak X\to\C\mid \exists C_f>0\ \forall x,y\in S^+\mathfrak X:\ |f(x)-f(y)|\le C_f d_\vartheta (x,y)\}$$
of Lipschitz continuous functions w.r.t. the metric $d_\vartheta$ we set 
$$|f|_\vartheta:= \inf \{C_f\mid \text{Lipschitz constants for } f\}$$
and $\|f\|_\vartheta:=|f|_\vartheta+\|f\|_\infty$. 
}
\end{definition} 

\begin{remark}\label{rem:Lipschitz-scale}{\rm
$(\mathcal F_\vartheta,\|\cdot\|_\vartheta)$ is a Banach space (see \cite[Exer.~1.16]{Ba00}) for each $\vartheta\in{}]0,1[{}$. If $0<\vartheta'<\vartheta<1$, then $C_c^\mathrm{lc}(S^+\mathfrak X)\subseteq \mathcal F_{\vartheta'}\subseteq \mathcal F_\vartheta$.  The spaces $\mathcal F_\vartheta$ can thus be seen as spaces with increasing regularity (for $\vartheta\to 0$). The locally constant functions are contained in all of them.
}
\end{remark}
\begin{remark}\label{rem:Hoelder}{\rm
Instead of working with Lipschitz functions for the scale $d_\vartheta$ of metrics
$0<\vartheta<1$ one can also fix $0<\vartheta_0<1$. Then, for $\vartheta_0\leq \vartheta<1$ we can write $\vartheta=\vartheta_0^\alpha$ for some $0<\alpha<1$ and obtain $d_\vartheta(x,y)= d_{\vartheta_0}(x,y)^\alpha$. The spaces $\mathcal F_\vartheta$ then corresponds to the space of $\alpha$-H\"older continuous functions w.r.t. the metric $d_{\vartheta_0}$. We therefore call the spaces $\mathcal F_\vartheta$ Hölder spaces.}
\end{remark}

\begin{definition}{\rm
We denote the topological dual of $\mathcal F_\vartheta$ by $\mathcal F_\vartheta'$. 
}
\end{definition}

\begin{remark}{\rm
\begin{itemize}
\item[(i)]The dual spaces $\mathcal F_\vartheta'$ are again Banach spaces and Remark~\ref{rem:Lipschitz-scale} implies that 
$$\mathcal F_\vartheta'\subseteq \mathcal F_{\vartheta'}'\subseteq \mathcal K'(S^+\mathfrak X)$$
for $0<\vartheta'<\vartheta<1$. Note that these spaces show up in the spectral theory of transfer operators for subshifts of finite type (see e.g. \cite{Ba00}).
\item[(ii)]
The situation bears some similarity to the inclusion of Sobolev spaces:
$$H^k\supset H^{k'}\supset C^\infty\supset \mathcal A\quad\text{and}\quad  H^{-k}\subset H^{-k'}\subset \mathcal D'\subset \mathcal A'$$
for $0<k<k'<\infty$. 
\end{itemize}
}
\end{remark}

Recall from Remark~\ref{rem:K'S+X} that $\mathcal K'(S^+\mathfrak X)\cong C\big(\mathfrak X, \mathcal K'(\Omega)\big)$. Next, we identify the function spaces on $\Omega$ which correspond to $\mathcal F_\vartheta$ and $\mathcal F_\vartheta'$. 
We start with a lemma.

\begin{definition}{\rm
On the space 
$$\mathcal F_{o,\vartheta}(\Omega):=\{ f:\Omega\to\C\mid \exists C_f>0\  \forall \omega_1,\omega_2\in \Omega:\ |f(\omega_1)-f(\omega_2)|\le C_f d_{o,\vartheta} (\omega_1,\omega_2)\}$$
of Lipschitz continuous functions w.r.t. the metric $d_{o,\vartheta}$ we set 
$$|f|_{o,\vartheta}:= \inf \{C_f\mid \text{Lipschitz constants for } f\}$$
and $\|f\|_{o,\vartheta}:=|f|_{o,\vartheta}+\|f\|_\infty$. 
}
\end{definition} 

\begin{lemma}
$\mathcal F_{o,\vartheta}(\Omega)$ is a Banach space w.r.t. the norm $\|f\|_{o,\vartheta}$. Moreover, the equivalence class of the norms $\|f\|_{o,\vartheta}$ does not depend on the choice of the base point $o$. 
\end{lemma}

\begin{proof}
Note first that given $f,f'\in \mathcal F_{o,\vartheta}(\Omega)$ with Lipschitz constants $C_f$ and $C_{f'}$, then $\max\{C_f,C_{f'}\}$ is a Lipschitz constant for $f+f'$. Thus the estimate
\begin{eqnarray*}
\|f+f'\|_{o,\vartheta}
&=&\|f+f'\|_\infty+ |f+f'|_{o,\vartheta}\\
&\le& \|f\|_\infty+\|f'\|_\infty +\max\{|f|_{o,\vartheta},|f'|_{o,\vartheta}\}\\
&\le& \|f\|_\infty+\|f'\|_\infty +|f|_{o,\vartheta}+ |f'|_{o,\vartheta}\\
&\le& \|f\|_{o,\vartheta} +\|f'\|_{o,\vartheta}
\end{eqnarray*}
shows that $\|\cdot\|_{o,\vartheta}$ satisfies the triangle inequality. The other norm properties are clearly satisfied, so $\left(\mathcal F_{o,\vartheta}(\Omega),\|\cdot\|_{o,\vartheta}\right)$ is a normed space.

Completeness follows from a standard three epsilon argument: Let $(f_k)_{k\in\mathbb N}$ be a $\|\cdot\|_{o,\vartheta}$-Cauchy sequence, hence a $\|\cdot\|_\infty$-Cauchy sequence and a $|\cdot|_{o,\vartheta}$-Cauchy sequence. Let $f$ be the $\|\cdot\|_\infty$-limit of $(f_k)_{k\in\mathbb N}$. It suffices to show that $f$ is also the $\|\cdot\|_{o,\vartheta}$-limit of $(f_k)_{k\in\mathbb N}$. To this end we note that for $\omega,\omega'\in \Omega$ we have
\begin{eqnarray*}
|f(\omega)-f(\omega')|
&\le &|f(\omega)-f_k(\omega)|+|f_k(\omega)-f_k(\omega')|+|f_k(\omega')-f(\omega')|\\
&\le &2\| f-f_k\|_\infty+|f_k|_{o,\vartheta} d_{o,\vartheta}(\omega,\omega')\\
&\underset{k\to\infty}{\longrightarrow}& \lim_{k\to\infty}|f_k|_{o,\vartheta}\, d_{o,\vartheta}(\omega,\omega'),
\end{eqnarray*}
which implies $|f|_{o,\vartheta}\le \lim_{k\to\infty}|f_k|_{o,\vartheta}$. Thus we see that $f\in \mathcal F_{o,\vartheta}(\Omega)$ and $\|F\|_\vartheta\le \lim_{k\to\infty}\|F_k\|_\vartheta$. As $(f_k)_{k\in\mathbb N}$ is a  $|\cdot|_{o,\vartheta}$-Cauchy sequence, we find for $\epsilon>0$ a $k_0\in\mathbb N$ such that $|f_j-f_k|_{o,\vartheta}\le \epsilon$ for $j,k\ge k_0$. Writing $f-f_k=(f-f_j)+(f_j-f_k)$ for $j, k\ge k_0$ we have
\begin{eqnarray*}
|(f-f_k)(c)-(f-f_k)(c')|
&\le &2\|(f-f_j)\|_\infty+|f_j-f_k|_\vartheta\ d_{o,\vartheta}(\omega,\omega')\\
&\le &2\|(f-f_j)\|_\infty+\epsilon\ d_{o,\vartheta}(\omega,\omega')\\
&\underset{j\to\infty}{\longrightarrow}&\epsilon\ d_{o,\vartheta}(\omega,\omega').
\end{eqnarray*}
 Thus $|f-f_k|_\vartheta\le \epsilon$ and we have shown that $f$ is also the $\|\cdot\|_{o,\vartheta}$-limit of $(f_k)_{k\in\mathbb N}$.  
 
The equivalence of the norms associated with different base points follows from the equivalence of the corresponding metrics on $\Omega$ that was asserted in Lemma~\ref{lem:metric Omega}.
\end{proof}

\begin{remark}{\rm
Consider the subsets $\Omega_o(v)$ of $\Omega$ given as the endpoints of geodesic rays starting in $o$ and passing through $v\in\mathfrak X$. As these sets form a basis for the topology on $\Omega$, for $f\in C^\mathrm{lc}(\Omega)$ there exists an $N\in\N$ such that $f\vert_{\Omega_o(v)}$ is constant for each $v$ with $d(o,v)>N$. In other words, for $d(o,v)>N$ we have $|f(\omega_1)-f(\omega_2)|=0$ for $\omega_1,\omega_2\in\Omega_o(v)\}$. Consequently $C^\mathrm{lc}(\Omega)\subseteq \mathcal F_{o,\vartheta}(\Omega)$ for all $0<\vartheta <1$ and any choice of base point $o\in\mathfrak X$. 
}
\end{remark}

\begin{lemma}
$\mathcal F_\vartheta(S^+\mathfrak X)\cong C\big(\mathfrak X, \mathcal F_{o,\vartheta}(\Omega)\big)$ as topological vector spaces, where the right hand side is equipped with the Banach norm 
$$\|\tilde f\|_{C(\mathfrak X,\mathcal F_\vartheta)}:=\sup_{x\in \mathfrak X}\|\tilde f(x)\|_{x,\vartheta}.$$
In particular the two norms are equivalent. 
\end{lemma}

\begin{proof} Recall the identification $S^+\mathfrak X\equiv \mathfrak X\times \Omega$ via $[x,\omega[{}\mapsto (x,\omega)$.  For $f: S^+\mathfrak X\equiv \mathfrak X\times \Omega\to \C$ we define $\tilde f:\mathfrak X\to \mathcal F(\Omega)$ via $\tilde f(x):=f(x,\cdot)$, where $\mathcal F(\Omega)$ is the space of $\C$-valued functions on $\Omega$.\\

\noindent
{\bf Claim:} $\|f\|_\vartheta\le 3\|\tilde f\|_{C(\mathfrak X, \mathcal F_\vartheta)}$.\\

\noindent
It is clear that $\|f\|_\infty=\sup_{x\in\mathfrak X} \|\tilde f(x)\|_\infty\le \|\tilde f\|_{C(\mathfrak X, \mathcal F_\vartheta)}$. Moreover,
\begin{align*}
|f|_\vartheta
  &\leq \max\left\{
    \begin{array}{ll}
      \displaystyle\sup_{x\in\mathfrak X}
      \inf\bigg\{C_f\,\,\mid\,\,
      \begin{array}{ll}
        \forall (x,\gamma),(x,\gamma')\in S^+_x\mathfrak X:\\[2mm]
        |f(x,\gamma)-f(x,\gamma')|\le C_f
        d_\vartheta\big((x,\gamma),(x,\gamma')\big)
      \end{array}
      \bigg\},
      \\[7mm]
      \displaystyle\sup_{(x,\gamma),(x',\gamma')\in S^+\mathfrak X}|f(x,\gamma)-f(x',\gamma')|
    \end{array}
  \right\}\\
&\le \max\{\sup_{x\in\mathfrak X}|\tilde f(x)|_{x,\vartheta},
2\|f\|_\infty\}\\
&\le \max\{\|\tilde f\|_{C(\mathfrak X, \mathcal F_\vartheta)}, 2\|\tilde f\|_{C(\mathfrak X, \mathcal F_\vartheta)}\}\\
&=2\|\tilde f\|_{C(\mathfrak X, \mathcal F_\vartheta)}.
\end{align*}
This proves the claim. To conclude the proof it suffices to observe that $\|\tilde f(x)\|_\infty\le \|f\|_\infty$ and $|\tilde f(x)|_{x,\vartheta}\le |f|_\vartheta$, since this implies $\|\tilde f\|_{C(\mathfrak X, \mathcal F_\vartheta)}\le \|f\|_\vartheta$.
\end{proof}

\begin{lemma}\label{lemmma:HoelderOmegadual}
Suppose that $\mu\in\mathcal F_{o,\vartheta}'(\Omega)$. Then for each $K>\frac{1}{\vartheta}$ there exists $C>0$ such that
\begin{equation}\label{cond}
\forall v\in\mathfrak X:\quad |\mu\big(\Omega_o(v)\big)|\le C K^{d(o,v)}.
\end{equation}
  
Conversely assume that  $\mu\in \mathcal M_\mathrm{fa}(\Omega)$ and $K, C>0$ satisfy Condition~(\ref{cond}). Then $\mu$ extends to a continuous linear functional on $\mathcal F_{o,\vartheta}(\Omega)$ for every $\vartheta$ satisfying $0<\vartheta<\frac{1}{Kq_\mathrm{max}}$.
\end{lemma}

\begin{proof}
  For $v \neq o$, one finds:
  \[
    | \mathbf{1}_{\Omega_o(v)} |_{\vartheta} = \vartheta^{1-d(o,v)}
    \qquad\text{and}\qquad
    \| \mathbf{1}_{\Omega_o(v)} \|_{0,\vartheta} = 1 + \vartheta^{1-d(o,v)}
  \]
  Assuming that $\mu : \mathcal F_{o,\vartheta}(\Omega) \rightarrow \C$ is
  a bounded linear functional, there is a constant $c > 0$ such that
  \[
    | \mu( \mathbf{1}_{\Omega_o(v)}) |
    \leq c \| \mathbf{1}_{\Omega_o(v)} \|_{0,\vartheta}
    = c ( 1 + \vartheta^{1-d(o,v)} )
  \]
  However, for any $K>\vartheta^{-1}$, there is $C > 0$ with
  \[
    c ( 1 + \vartheta^{1-n} )
    \leq C K^{n}
    \qquad\text{for all\ }n
  \]
  whence~(\ref{cond}) holds.

  Now, we assume that $\mu\in \mathcal M_\mathrm{fa}(\Omega)$ and $K, C>0$
  satisfy the Condition~(\ref{cond}). We explicitly construct a continuous
  extension of $\mu$ to $\mathcal F_{o,\vartheta}(\Omega)$. Note that that the
  locally constant functions are not dense in $\mathcal F_{o,\vartheta}(\Omega)$, whence the extension might not be unique. We base our construction on a pre-chosen
  way to push vertices away from $o$ to infinity, i.e., a map $W:\mathfrak X\to\Omega$ such that
$$\forall v\in\mathfrak X: W(v)\in \Omega_o(v)$$
Given $f\in \mathcal F_{o,\vartheta}(\Omega)$, we define
$$\mu_{W,n}(f):= \sum_{d(o,v)=n} \mu\big(\Omega_o(v)\big)f\big(W(v)\big)$$
for $n\in\N$. Then we find
\begin{align*}
|\mu_{W,n}(f)- \mu_{W,n+1}(f)|
&\le \Big|\sum_{d(o,v)=n}\sum_{d(v',v)=1} \mu(\Omega_o(v'))|f\big(W(v)\big)-f\big(W(v')\big)|\Big|\\
&\le (q_\mathrm{max}+1)q_\mathrm{max}^nCK^{n+1}\|f\|_{o,\vartheta}\vartheta^n \longrightarrow 0
\end{align*}  
if $\vartheta< \frac{1}{Kq_\mathrm{max}}$. Here the summation over $v'$ extends over all neighbors of $v$ which are not in $[o,v]$. Thus, for  $\vartheta< \frac{1}{Kq_\mathrm{max}}$ the limit $\mu_W(f):=\lim_{n\to \infty} \mu_{W,n}(f)$ exists and satisfies (use geometric series)
$$|\mu_W(f)|\le (q+1)CK\|f\|_{o,\vartheta}\frac{1}{1-K\vartheta q_\mathrm{max}}.$$
Next we observe that $\mu_W(f)$ is actually independent of the choice of $W$. In fact, let $W'$ be another such function. Then 
\begin{align*}
|\mu_{W,n}(f)-\mu_{W',n}(f)|
&\le \Big|\sum_ {d(o,v)=n} \mu(\Omega_o(v))|f(W(v))-f(W'(v))|\Big|\\
&\le (q_\mathrm{max}+1)q_\mathrm{max}^{n-1} CK^n\|f\|_{o,\vartheta} \vartheta^n\underset{n\to\infty}{\longrightarrow} 0.
\end{align*}  
In order to conclude the proof we have to show that $\mu_W\in\mathcal K'(\Omega)$ agrees with $\mu$ when viewed as a finitely additive measure. It suffices to show that for any  $v_0\in \mathfrak X$ we have $\mu_W(\Omega_o(v_0))=\mu(\Omega_o(v_0))$. We have
\begin{eqnarray*}
\mu_W(\Omega_o(v_0))
&=&\mu_W(\mathbf 1_{\Omega_o(v_0)})
\ =\ \lim_{n\to\infty}\mu_{W,n}(\mathbf 1_{\Omega_o(v_0)})\\
&=&\lim_{n\to\infty}\sum_{d(o,v)=n}\mu(\Omega_o(v))\mathbf 1_{\Omega_o(v_0)}(W(v)).
\end{eqnarray*}
Note that for $n\ge d(o,v_0)$ precisely the $v$ with $v_0\in[o,v]$ contribute to the sum, which is then equal to 
$$\sum_{d(o,v)=n, v_0\in[o,v]}\mu(\Omega_o(v))= \mu(\Omega_o(v_0)).$$
Thus the sequence is stationary beyond $d(o,v_0)$ with limit $\mu(\Omega_o(v_0))$. 
\end{proof}

\begin{definition} 
{\rm
We say that $g\in C(\mathfrak X)$ is  of \emph{moderate growth} if there exists $B,G>0$ such that $|g(x)|\le BG^{d(o,x)}$ for all $x\in \mathfrak X$.
}
\end{definition}

Our final regularity theorem is now basically a corollary to  Theorem~\ref{thm:Poisson}.

\begin{theorem}\label{thm:mod-growth}
Let $z^2\not\in\{0,1\}$. Then a function $f\in \mathcal E_{\chi(z)}(\mathfrak X)$ is of moderate growth if and only if the boundary value $\mu=\vec\beta_z(f)\in\mathcal M_\mathrm{fa}(\Omega)=\mathcal K'(\Omega)$ is contained in $\mathcal F_{o,\vartheta}'(\Omega)$ for some $\vartheta>0$.
\end{theorem}
\begin{proof}
  By Lemma~\ref{lemmma:HoelderOmegadual} it suffices to show that the following are
  equivalent:
  \begin{enumerate}
    \item
      There exist
      $B,G>0$ with $|f(x)|\le BG^{d(o,x)}$ for all $x\in \mathfrak X$.
    \item
      There exist
      $C,K>0$ with $|\mu\big(\Omega_o(x)\big)|\le C K^{d(o,x)}$ for all $x\in \mathfrak X$.
  \end{enumerate}
  As $f$ and $\mu$ satisfy Condition~(\ref{beta}) and $\vec\mu(\vec e)=\mu\left(\Omega_o(\tau(\vec e))\right)$ for any $\vec e\in {\mathfrak E}_o$, this equivalence follows by a straightforward calculation.
\end{proof}

\newcommand{\etalchar}[1]{$^{#1}$}
\providecommand{\bysame}{\leavevmode\hbox to3em{\hrulefill}\thinspace}
\providecommand{\MR}{\relax\ifhmode\unskip\space\fi MR }
\providecommand{\MRhref}[2]{%
  \href{http://www.ams.org/mathscinet-getitem?mr=#1}{#2}
}
\providecommand{\href}[2]{#2}

\end{document}